\documentclass[12]{article}
\usepackage{amsmath}
\usepackage{amssymb}
\usepackage{amsfonts}

\newtheorem{theorem}{Theorem}[section]

\newtheorem{corollary}[theorem]{Corollary}

\newtheorem{definition}[theorem]{Definition}
\newtheorem{example}[theorem]{Example}

\newtheorem{lemma}[theorem]{Lemma}

\newtheorem{proposition}[theorem]{Proposition}
\newtheorem{remark}[theorem]{Remark}

\newenvironment{proof}[1][Proof]{\noindent\textbf{#1.} }{\ \rule{0.5em}{0.5em}}
\input{tcilatex}
\begin{document}

\title{A Sheaf On The Second Spectrum Of A Module\thanks{%
This paper is submitted to Communications in Algebra on June 15th, 2017 for
the referee process.}}
\author{Se\c{c}il \c{C}eken and Mustafa Alkan}
\date{cekensecil@gmail.com, alkan@akdeniz.edu.tr}
\maketitle

\begin{abstract}
Let $R$ be a commutative ring with identity and $Spec^{s}(M)$ denote the set
all second submodules of an $R$-module $M$. In this paper, we construct and
study a sheaf of modules, denoted by $\mathcal{O}(N,M)$, on $Spec^{s}(M)$
equipped with the dual Zariski topology of $M$, where $N$ is an $R$-module.
We give a characterization of the sections of the sheaf $\mathcal{O}(N,M)$
in terms of the ideal transform module. We present some interrelations
between algebraic properties of $N$ and the sections of $\mathcal{O}(N,M)$.
We obtain some morphisms of sheaves induced by ring and module homomorphisms.

\textbf{2010 Mathematics Subject Classification: }13C13, 13C99, 14A15, 14A05.

\textbf{Keywords and phrases: }Second submodule, dual Zariski topology,
sheaf of modules.
\end{abstract}

\section{Introduction}

Throughout this article all rings will be commutative rings with identity
elements and all modules will be unital left modules. Unless otherwise
stated $R$ will denote a ring. Given an $R$-module $M$, the annihilator of $M
$ (in $R$) is denoted by $ann_{R}(M)$ and for an ideal $I$ of $R$, the
annihilator of $I$ in $M$ is defined as the set $(0:_{M}I):=\{m\in M:Im=0\}$%
. Clearly, $(0:_{M}I)$ is a submodule of $M$.

Recall that a \textit{sheaf of rings (modules)} $\mathcal{F}$ on a
topological space $X$ is an assignment of a ring (module) $\mathcal{F}(U)$
to each open subset $U$ of $X$, together with, for each inclusion of open
subsets $V\subseteq U$, a morphism of rings (modules) $\rho _{UV}:\mathcal{F}%
(U)\longrightarrow \mathcal{F}(V)$ subject to the following conditions:

$(i)$ $\mathcal{F}(\emptyset )=0$.

$(ii)$ $\rho _{UU}=id_{\mathcal{F(}U)}$.

$(iii)$ If $W\subseteq V\subseteq U$, then $\rho _{UW}=\rho _{VW}o\rho _{UV}$%
.

$(iv)$ If $U$ is an open subset of $X$ and $\{U_{\alpha }\}_{\alpha \in
\Lambda }$ is an open cover of $U$, and if $f\in \mathcal{F}(U)$ is an
element such that $\rho _{UU_{\alpha }}(f)=0$ for all $\alpha \in \Lambda $,
then $f=0$.

$(v)$ If $U$ is an open subset of $X$ and $\{U_{\alpha }\}_{\alpha \in
\Lambda }$ is an open cover of $U$, and if we have a collection elements $%
f_{\alpha }\in \mathcal{F}(U_{\alpha })$ with the property that for all $%
\alpha $, $\beta \in \Lambda $, $\rho _{U_{\alpha }\left( U_{\alpha }\cap
U_{\beta }\right) }(f_{\alpha })=\rho _{U_{\beta }(U_{\alpha }\cap U_{\beta
})}(f_{\beta })$, then there is an element $f\in \mathcal{F}(U)$ such that
for all $\alpha $, $f_{\alpha }=\rho _{UU_{\alpha }}(f)$.

If $\mathcal{F}$ is a sheaf on a topological space $X$, we refer to $%
\mathcal{F}(U)$ as the \textit{sections} of $\mathcal{F}$ over the open
subset $U$. We call the maps $\rho _{UV}$ as the \textit{restriction maps }%
(cf. \cite{Hartshone}).

The prime spectrum of a ring $R$, denoted by $Spec(R)$, consists of all
prime ideals of $R$ and it is non-empty. For each ideal $I$ of $R$, the sets 
$V(I)=\{p\in Spec(R):I\subseteq p\}$, where $I$ is an ideal of $R$, satisfy
the axioms for closed sets of a topology on $Spec(R)$, called the \textit{%
Zariski topology of} $R$.

It is well-known that for any commutative ring $R$, there is a sheaf of
rings on $Spec(R)$, denoted by $\mathcal{O}_{Spec(R)}$, defined as follows:
For an open subset $U\subseteq Spec(R)$, we define $\mathcal{O}_{Spec(R)}(U)$
to be the set of all functions $s:U\rightarrow \tbigsqcup\nolimits_{p\in
U}R_{p}$, such that $s(p)\in R_{p}$, for each $p\in U$, and such that for
each $p\in U$, there is a neighborhood $V$ of $p$, contained in $U$, and
elements $a$, $f\in R$, such that for each $q\in V$, we have $f\not\in q$
and $s(q)=\frac{a}{f}$ in $R_{q}$ (see \cite{Hartshone}).

Let $M$ be an $R$-module. A proper submodule $N$ of $M$ is said to be 
\textit{prime }if for any $r\in R$ and $m\in M$ with $rm\in N$, we have $%
m\in N$ or $r\in ann_{R}(M/N):=(N:M)$. If $N$ is a prime submodule of $M$,
then $p=(N:M)$ is a prime ideal of $R$. In this case, $N$ is called a $p$%
-prime submodule of $M$. The set of all prime submodules of a module $M$ is
called the \textit{prime spectrum} of $M$ and denoted by $Spec(M)$. For any
submodule $N$ of an $R$-module $M$, we have a set $V(N)=\{P\in
Spec(M):(N:M)\subseteq (P:M)\}$. Then the sets $V(N)$, where $N$ is a
submodule of $M$ satisfy the axioms for closed sets of a topology on $%
Spec(M) $, called the \textit{Zariski topology of }$M$. Several authors have
investigated the prime spectrum and the Zariski topology of a module over
the last twenty years (see for example \cite{Lu-95}, \cite{Lu}, \cite%
{Lu-2007}, \cite{Lu-Noether}, \cite{McCasland-Smith}). Recently, some
authors have investigated a sheaf structure on the prime spectrum of a
module which generalizes the sheaf of rings $\mathcal{O}_{Spec(R)}$ on the
topological space $Spec(R)$. In \cite{U.Tekir}, the author obtained an $R$%
-module $\mathcal{O}_{Spec(M)}(U)$ for each open subset $U$ of $Spec(M)$
equipped with the Zariski topology of $M$ such that $\mathcal{O}_{Spec(M)}$
is a sheaf of modules on $Spec(M)$. In \cite{Ab-Le-2}, the authors defined
and studied a sheaf of modules which is denoted by $\mathcal{A}(N,M)$ on the
topological space $Spec(M)$ equipped with the Zariski topology of $M$, where 
$M$ and $N$ are two $R$-modules. In fact, both $\mathcal{O}_{Spec(M)}$ and $%
\mathcal{A}(N,M)$ are generalizations of the sheaf of rings $\mathcal{O}%
_{Spec(R)}$ to modules. In \cite{Ab-Le-2}, the authors proved that if $N=R$,
then $\mathcal{A}(R,M)$ is a scheme on $Spec(M)$. This scheme structure were
investigated in \cite{Ab-Le-1}.

Recently, a dual theory of prime submodules has been developed and
extensively studied by many authors. The dual notion of prime submodules was
first introduced by S. Yassemi in \cite{Yassemi}. A submodule $N$ of an $R$%
-module $M$ is said to be a \textit{second submodule} provided $N\neq 0$
and, for all $r\in R$, $rN=0$ or $rN=N$. If $N$ is a second submodule of $M$%
, then $p=ann_{R}(N)$ is a prime ideal of $R$. In this case, $N$ is called a 
$p$-second submodule of $M$ (cf. \cite{Yassemi}). In recent years, second
submodules have attracted attention of various authors and they have been
studied in a number of papers (see for example \cite{Abuhlail}, \cite%
{sec-colloquium}, \cite{Sec2}, \cite{Second rad}, \cite{APK}, \cite{CAS1}, 
\cite{CAS2}, \cite{CA}, \cite{CA-AMS}).

The set of all second submodules of a module $M$ is called the \textit{%
second spectrum} of $M$ and denoted by $Spec^{s}(M)$. As in \cite%
{dual-zariski}, for any submodule $N$ of an $R$-module $M$ we define $%
V^{s\ast }(N)$ to be the set of all second submodules of $M$ contained in $N$%
. Clearly $V^{s\ast }(0)$ is the empty set and $V^{s\ast }(M)$ is $%
Spec^{s}(M)$. Note that for any family of submodules $N_{i}$ $(i\in I)$ of $%
M $, $\cap _{i\in I}V^{s\ast }(N_{i})=V^{s\ast }(\cap _{i\in I}N_{i})$. Thus
if $Z^{s\ast }(M)$ denotes the collection of all subsets $V^{s\ast }(N)$ of $%
Spec^{s}(M)$ where $N\leq M$, then $Z^{s\ast }(M)$ contains the empty set
and $Spec^{s}(M),$ and $Z^{s\ast }(M)$ is closed under arbitrary
intersections. But in general $Z^{s\ast }(M)$ is not closed under finite
unions. A module $M$ is called a \textit{cotop module} if $Z^{s\ast }(M)$ is
closed under finite unions. In this case $Z^{s\ast }(M)$ is called the 
\textit{quasi-Zariski topology} on $Spec^{s}(M)$ (see \cite{ATF1}). Note
that in \cite{Abuhlail} a cotop module was called a top$^{\text{s}}$-module.
More information about the class of cotop modules can be found in \cite%
{Abuhlail} and \cite{ATF1}.

Let $M$ be an $R$-module and $N$ be a submodule of $M$. We define the set $%
V^{s}(N):=\{S\in Spec^{s}(M):ann_{R}(N)\subseteq ann_{R}(S)\}$. In \cite[%
Lemma 2]{dual-zariski}, it was shown that $%
V^{s}(N)=V^{s}((0:_{M}ann_{R}(N)))=V^{s\ast }((0:_{M}ann_{R}(N)))$, in
particular, $V^{s}((0:_{M}I))=V^{s\ast }((0:_{M}I))$ for every ideal $I$ of $%
R$, and that the set $Z^{s}(M)=\{V^{s}(N):N\leq M\}$ satisfies the axioms
for the closed sets. Thus there exists a topology, say $\tau ^{s}$, on $%
Spec^{s}(M)$ having $Z^{s}(M)$ as the family of closed subsets. This
topology is called the dual Zariski topology of $M$ (see \cite[Lemma 2]%
{dual-zariski}). Dual Zariski topology, the second spectrum of modules, and
related notions have been investigated by some authors in recent years (see 
\cite{Abuhlail3}, \cite{Abuhlail}, \cite{ATF1}, \cite{ATF2}, \cite%
{dual-zariski}, \cite{CA-JAA}, \cite{CA-Malezya-2017} and \cite{Farshadifar}%
).

In this paper, we define and study a sheaf structure on the second spectrum
of a module. Let $M$ be an $R$-module. In Section 2, we construct a sheaf,
denoted by $\mathcal{O}(N,M)$, on $Spec^{s}(M)$ equipped with the dual
Zariski topology of $M$, where $N$ is an $R$-module. Firstly, we find the
stalk of the sheaf $\mathcal{O}(N,M)$ (see Theorem \ref{Stalk}). In Theorem %
\ref{Theorem8}, we give a characterization for the sections of the sheaf $%
\mathcal{O}(N,M)$ in terms of the ideal transform module. Let $R$ be a
Noetherian ring and $M$ be a faithful secondful $R$-module. We prove that if 
$N$ is a free, projective or flat $R$-module, then so is $\mathcal{O}%
(N,M)(Spec^{s}(M))$ (see Theorem \ref{Free-O(N,M)(X)}). In Section 3, we
deal with a scheme structure on the second spectrum of a module. In Theorem %
\ref{Scheme}, we prove that $\mathcal{O}(R,M)$ is a scheme when $M$ is a
faithful secondful $R$-module and $Spec^{s}(M)$ is a $T_{0}$-space. Then we
define two morphisms of locally ringed spaces by using ring and module
homomorphisms (see Theorem \ref{morphism1} and Corollary \ref{morphism2}).

\section{A Sheaf Structure On The Second Spectrum Of A Module}

Throughout the rest of the paper $M$ will be an $R$-module, $X^{s}$ will
denote $Spec^{s}(M)$ and we consider $X^{s}$ with the dual Zariski topology
unless otherwise stated. For every open subset $U$ of $X^{s}$, we set $%
Supp^{s}(U)=\{ann_{R}(S):S\in U\}$.

In this section, we construct a sheaf on $X^{s}$ and investigate some
properties of this sheaf.

\begin{definition}
Let $N$ be an $R$-module. For every open subset $U$ of $X^{s}$ we define $%
\mathcal{O}(N,M)(U)$ to be the set of all elements $(\beta _{p})_{p\in
Supp^{s}(U)}\in \tprod\limits_{p\in Supp^{s}(U)}N_{p}$ in which for each $%
Q\in U$, there is an open neighborhood $W$ of $Q$ with $Q\in W\subseteq U$
and there exist elements $t\in R$, $m\in N$ such that for every $S\in W$, we
have $t\not\in p:=ann_{R}(S)$ and $\beta _{p}=\frac{m}{t}\in N_{p}$.
\end{definition}

Let $U$ and $V$ be open subsets of $X^{s}$ with $\emptyset \neq V\subseteq U$
and $\beta =(\beta _{p})_{p\in Supp^{s}(U)}\in \mathcal{O}(N,M)(U)$. Then it
is clear that the restriction $(\beta _{p})_{p\in Supp^{s}(V)}$ belongs to $%
\mathcal{O}(N,M)(V)$. Therefore, we have the restriction map $\rho _{UV}:%
\mathcal{O}(N,M)(U)\longrightarrow \mathcal{O}(N,M)(V)$, $\rho _{UV}(\beta
)=(\beta _{p})_{p\in Supp^{s}(V)}$ for all $\beta =(\beta _{p})_{p\in
Supp^{s}(U)}\in \mathcal{O}(N,M)(U)$. We define $\rho _{U\emptyset }:%
\mathcal{O}(N,M)(U)\longrightarrow \mathcal{O}(N,M)(\emptyset )=0$ to be the
zero map. It is clear from the local nature of the definition, $\mathcal{O}%
(N,M)$ is a sheaf with the restriction maps defined above. We can define a
map $\tau _{N}^{U}:N\longrightarrow \mathcal{O}(N,M)(U)$ by $\tau
_{N}^{U}(n)=(\frac{n}{1})_{p\in Supp^{s}(U)}$ for all $n\in N$. We note that 
$\tau _{N}^{U}$ is an $R$-module homomorphism. Clearly, $\rho _{UV}$ $o$ $%
\tau _{N}^{U}=\tau _{N}^{V}$.

Recall that for any $r\in R$, the set $D_{r}:=Spec(R)\backslash V(Rr)$ is
open in $Spec(R)$ and the family $\{D_{r}:r\in R\}$ forms a base for the
Zariski topology on $Spec(R)$. Let $M$ be an $R$-module. For each $r\in R$,
we define $Y_{r}:=X^{s}\backslash V^{s}((0:_{M}r))$. In \cite[Theorem 4]%
{dual-zariski}, it was shown that the set $B=\{Y_{r}:r\in R\}$ forms a base
for the dual Zariski topology on $X^{s}$.

\begin{remark}
Let $r\in R$ and $Q\in X^{s}$. Then, $r\not\in ann_{R}(Q)$ if and only if $%
Q\in Y_{r}$. Let $Q\in Y_{r}$. Suppose that $r\in ann_{R}(Q)$, then $rQ=0$
and so $Q\subseteq (0:_{M}r)$. This implies that $Q\in V^{s\ast
}((0:_{M}r))=V^{s}((0:_{M}r))$, a contradiction. Conversely, let $r\not\in
ann_{R}(Q)$. Then $Q\not\subseteq (0:_{M}r)$ and so $Q\in Y_{r}$. In the
proof our results, we will use this fact without any further comment.
\end{remark}

Let $\mathcal{F}$ be a sheaf of modules (rings) on a topological space $X$
and $P\in X$. Recall that the stalk $\mathcal{F}_{P}$ of $\mathcal{F}$ at $P$
is defined to be the direct limit $\underset{\overrightarrow{P\in U}}{\lim }%
\mathcal{F}(U)$ of the modules (rings) $\mathcal{F}(U)$ for all open subsets 
$U$ of $X$ containing $P$ via the restriction maps (see \cite{Hartshone}).
In the following theorem, we determine the stalk of the sheaf $\mathcal{O}%
(N,M)$ at a point $S$ in $X^{s}$.

\begin{theorem}
\label{Stalk}Let $N$ be an $R$-module and $S\in X^{s}$. Then, the stalk $%
\mathcal{O}(N,M)_{S}$ of the sheaf $\mathcal{O}(N,M)$ at $S$ is isomorphic
to $N_{p}$ where $p:=ann_{R}(S)$.
\end{theorem}

\begin{proof}
Let $S$ be a $p$-second submodule of $M$ and $m\in \mathcal{O}(N,M)_{S}=%
\underset{\overrightarrow{P\in U}}{\lim }\mathcal{O}(N,M)(U)$. Then there
exists an open neighborhood $U$ of $S$ and $\beta =(\beta _{p})_{p\in
Supp^{s}(U)}\in \mathcal{O}(N,M)(U)$ such that $\beta $ represents $m$. We
define $\phi :\mathcal{O}(N,M)_{S}\longrightarrow N_{p}$ by $\phi (m)=\beta
_{p}$. Let $V$ be another neighborhood of $S$ and $\alpha =(\alpha
_{p})_{p\in Supp^{s}(V)}\in \mathcal{O}(N,M)(V)$ such that $\alpha $ also
represents $m$. Then there exists an open set $W\subseteq U\cap V$ such that 
$S\in W$ and $\beta _{_{\mid W}}=\alpha _{_{\mid W}}$. Since $S\in W$, we
have $\alpha _{p}=\beta _{p}$. This shows that $\phi $ is a well-defined
map. We claim that $\phi $ is an isomorphism.

Let $x\in N_{p}$. Then $x=\frac{a}{t}$ for some $a\in N$, $t\in R\backslash
p $. Since $t\not\in p=ann_{R}(S)$, we have $S\in Y_{t}$. Now we define $%
\beta _{q}=\frac{a}{t}$ in $N_{q}$ for all $Q\in Y_{t}$ where $t\not\in
q=ann_{R}(Q)$. Then $\beta =(\beta _{q})_{q\in Supp^{s}(Y_{t})}\in \mathcal{O%
}(N,M)(Y_{t})$. If $m$ is the equivalence class of $\beta $ in $\mathcal{O}%
(N,M)_{S}$, then $\phi (m)=x$. Hence $\phi $ is surjective.

Now let $m\in \mathcal{O}(N,M)_{S}$ and $\phi (m)=0$. Let $U$ be an open
neeighborhood of $S$ and $\beta =(\beta _{p})_{p\in Supp^{s}(U)}\in \mathcal{%
O}(N,M)(U)$ is a representative of $m$. There is an open neighborhood $%
V\subseteq U$ of $S$ and there are elements $a\in N$, $t\in R$ such that for
all $Q\in V$, we have $t\not\in q:=ann_{R}(Q)$ and $\beta _{q}=\frac{a}{t}%
\in N_{q}$. Then $0=\phi (m)=\beta _{p}=\frac{a}{t}$ in $N_{p}$. So there is 
$h\in R\backslash p$ such that $ha=0$. For all $Q\in Y_{th}$, we have $\beta
_{q}=\frac{ha}{ht}=\frac{a}{t}=0$ in $N_{q}$ where $q=ann_{R}(Q)$. Thus $%
\beta _{|Y_{th}}=0$. Therefore, $\beta =0$ in $\mathcal{O}(N,M)(Y_{th})$.
Consequently, $m=0$. This shows that $\phi $ is injective. Thus $\phi $ is
an isomorphism.
\end{proof}

A \textit{ringed space} is a pair $(X,\mathcal{O}_{X})$ consisting of a
topological space $X$ and a sheaf of rings $\mathcal{O}_{X}$ on $X$. The
ringed space $(X,\mathcal{O}_{X})$ is called a \textit{locally ringed space}
if for each point $P\in X$, the stalk $\mathcal{O}_{X,P}$ is a local ring
(cf. \cite{Hartshone}).

\begin{corollary}
$(X^{s},\mathcal{O}(R,M))$ is a locally ringed space.
\end{corollary}

\begin{example}
Consider the $%
\mathbb{Z}
$-modules $M=%
\mathbb{Q}
\oplus 
\mathbb{Z}
_{p}$ and $N=%
\mathbb{Z}
$, where $p$ is a prime number. Then $Spec^{s}(M)=\{%
\mathbb{Q}
\oplus 0$, $0\oplus 
\mathbb{Z}
_{p}\}$. By Theorem \ref{Stalk},

\begin{eqnarray*}
\mathcal{O}(%
\mathbb{Z}
,%
\mathbb{Q}
\oplus 
\mathbb{Z}
_{p})_{%
\mathbb{Q}
\oplus 0} &\simeq &%
\mathbb{Z}
_{(0)}=%
\mathbb{Q}
\text{ and} \\
\mathcal{O}(%
\mathbb{Z}
,%
\mathbb{Q}
\oplus 
\mathbb{Z}
_{p})_{0\oplus 
\mathbb{Z}
_{p}} &\simeq &%
\mathbb{Z}
_{p%
\mathbb{Z}
}=%
\mathbb{Z}
_{(p)}=\{\frac{a}{b}\in 
\mathbb{Q}
:a,b\in 
\mathbb{Z}
\text{, }b\neq 0\text{, }p\nmid b\}\text{.}
\end{eqnarray*}
\end{example}

Let $M$ be an $R$-module. The map $\psi ^{s}:Spec^{s}(M)\longrightarrow
Spec(R/ann_{R}(M))$ defined by $\psi ^{s}(S)=ann_{R}(S)/ann_{R}(M)$ is
called \textit{the natural map of }$Spec^{s}(M)$\textit{. }$M$ is said to be 
\textit{secondful} if the natural map $\psi ^{s}$ is surjective (cf. \cite%
{Farshadifar}).

Let $M$ be an $R$-module. The \textit{Zariski socle} of a submodule $N$ of $%
M $, denoted by $Z.soc(N)$, is defined to be the sum of all members of $%
V^{s}(N)$ and if $V^{s}(N)=\emptyset $, then $Z.soc(N)$ is defined to be $0$
(cf. \cite{Farshadifar}).

\begin{lemma}
\label{Lemma2}Let $R$ be a Noetherian ring and $N$ be an $R$-module. Let $M$
be a secondful $R$-module and $U=X^{s}\backslash V^{s}(K)$ where $K\leq M$.
Then for each $\beta =(\beta _{p})_{p\in Supp^{s}(U)}\in \mathcal{O}(N,M)(U)$%
, there exist $r\in 
\mathbb{Z}
^{+},$ $s_{1},...,s_{r}\in ann_{R}(K)$ and $m_{1},...,m_{r}\in N$ such that $%
U=\cup _{i=1}^{r}Y_{s_{i}}$ and $\beta _{p}=\frac{m_{i}}{s_{i}}$ for all $%
S\in Y_{s_{i}}$, $i=1,...,r$ where $p=ann_{R}(S)$.
\end{lemma}

\begin{proof}
Since $R$ is Noetherian, $U$ is quasi-compact by \cite[Corollary 4.4-(d)]%
{Farshadifar}. Thus there exist $n\in 
\mathbb{Z}
^{+}$, open subsets $W_{1},...,W_{n}$ of $U$, $t_{1},...,t_{n}\in R$, $%
a_{1},...,a_{n}\in N$ such that $U=\cup _{j=1}^{n}W_{j}$ and for each $%
j=1,...,n$ and $S\in W_{j}$ we have $\beta _{p}=\frac{a_{j}}{t_{j}}$ where $%
t_{j}\not\in p:=ann_{R}(S)$. Fix $j\in \{1,...,n\}$. There is a submodule $%
H_{j}$ of $M$ such that $W_{j}=X^{s}\backslash V^{s}(H_{j})$. Also we have $%
V^{s}(K)\subseteq V^{s}(H_{j})$. Since $R$ is Noetherian, $%
ann_{R}(H_{j})=Rb_{j1}+...+Rb_{jn_{j}}$ for some $b_{j1},...,b_{jn_{j}}\in R$%
. This implies that

$W_{j}=X^{s}\backslash V^{s}(H_{j})=X^{s}\backslash V^{s\ast
}((0:_{M}ann_{R}(H_{j})))$

$\ \ \ \ \ =X^{s}\backslash V^{s\ast
}((0:_{M}Rb_{j1}+...+Rb_{jn_{j}}))=X^{s}\backslash V^{s\ast }(\cap
_{f=1}^{n_{j}}(0:_{M}Rb_{jf}))$

\ \ \ \ \ $=X^{s}\backslash \cap _{f=1}^{n_{j}}V^{s\ast
}((0:_{M}Rb_{jf}))=\cup _{f=1}^{n_{j}}(X^{s}\backslash V^{s\ast
}((0:_{M}Rb_{jf})))$

\ \ \ \ \ $=\cup _{f=1}^{n_{j}}(X^{s}\backslash V^{s}((0:_{M}Rb_{jf})))=\cup
_{f=1}^{n_{j}}Y_{b_{jf}}$, i.e. $W_{j}=\cup _{f=1}^{n_{j}}Y_{b_{jf}}$.

On the other hand, we have $Z.soc(K)\subseteq Z.soc(H_{j})$ as $%
V^{s}(K)\subseteq V^{s}(H_{j})$. This implies that $ann_{R}(Z.soc(H_{j}))%
\subseteq ann_{R}(Z.soc(K))$. By \cite[Theorem 3.5-(e)]{Farshadifar}, we get 
$\sqrt{ann_{R}(H_{j})}\subseteq \sqrt{ann_{R}(K)}$. Since $R$ is Noetherian,
there exists $d\in 
\mathbb{Z}
^{+}$ such that $(\sqrt{ann_{R}(K)})^{d}\subseteq ann_{R}(K)$ and we have%
\begin{equation*}
ann_{R}(H_{j})^{d}\subseteq (\sqrt{ann_{R}(H_{j})})^{d}\subseteq (\sqrt{%
ann_{R}(K)})^{d}\subseteq ann_{R}(K)
\end{equation*}%
It follows that $b_{jf}^{d}\in ann_{R}(K)$ for each $f\in \{1,...,n_{j}\}$.
Also, for each $f\in \{1,...,n_{j}\}$ and for each $S\in Y_{b_{jf}}$, we
have $t_{j}b_{jf}^{d}\not\in ann_{R}(S)$. We conclude that 
\begin{equation*}
Y_{b_{jf}}=X^{s}\backslash V^{s}((0:_{M}b_{jf}))=X^{s}\backslash
V^{s}((0:_{M}t_{j}b_{jf}^{d}))=Y_{t_{j}b_{jf}^{d}}
\end{equation*}%
and we can write $\beta _{ann_{R}(S)}=\frac{a_{j}}{t_{j}}=\frac{%
a_{j}b_{jf}^{d}}{t_{j}b_{jf}^{d}}\in N_{ann_{R}(S)}$. This completes the
proof.
\end{proof}

Let $K$ be an $R$-module. For an ideal $I$ of $R$, the $I$\textit{-torsion
submodule} of $K$ is defined to be $\Gamma _{I}(K):=\cup _{n\geq
1}(0:_{K}I^{n})$ and $K$ is said to be $I$\textit{-torsion }if $K=\Gamma
_{I}(K)$ (cf. \cite{Broadman-Sharp}).

\begin{lemma}
\label{Lemma3}Let $N$ be an $R$-module and $U=X^{s}\backslash V^{s}(K)$
where $K\leq M$. Then $\Gamma _{ann_{R}(K)}(\mathcal{O}(N,M)(U))=0$.
\end{lemma}

\begin{proof}
Let $\beta =(\beta _{p})_{p\in Supp^{s}(U)}\in \Gamma _{ann_{R}(K)}(\mathcal{%
O}(N,M)(U))$. There exists $n\in 
\mathbb{Z}
^{+}$ such that $ann_{R}(K)^{n}\beta =0$. Consider $p\in Supp^{s}(U)$. There
exists a second submodule $Q\in U$ such that $p=ann_{R}(Q)$. If $%
ann_{R}(K)^{n}\subseteq p$, then $ann_{R}(K)\subseteq p$ and so $Q\in
V^{s}(K)$, a contradiction. Thus $ann_{R}(K)^{n}\not\subseteq p$. So there
exists $t_{p}\in ann_{R}(K)^{n}\backslash p$. Since $t_{p}\beta =0,$ we have 
$t_{p}\beta _{p}=0$ and $\beta _{p}=\frac{t_{p}\beta _{p}}{t_{p}}=0\in N_{p}$%
. Hence, $\beta =0$.
\end{proof}

\begin{theorem}
\label{Theorem4}Let $R$ be a Noetherian ring, $M$ be a faithful secondful $R$%
-module and $N$ be an $R$-module. Let $U=X^{s}\backslash V^{s}(K)$ where $%
K\leq M$. Then $\Gamma _{ann_{R}(K)}(N)=ker(\tau _{N}^{U})$, and so $%
ker(\tau _{N}^{U})$ is $ann_{R}(K)$-torsion.
\end{theorem}

\begin{proof}
By Lemma \ref{Lemma3}, we have $\tau _{N}^{U}(\Gamma
_{ann_{R}(K)}(N))\subseteq \Gamma _{ann_{R}(K)}(\mathcal{O}(N,M)(U))=0$. So $%
\Gamma _{ann_{R}(K)}(N)\subseteq ker(\tau _{N}^{U})$.

Suppose that $m\in ker(\tau _{N}^{U})$. Then $\frac{m}{1}=0\in N_{p}$ for
all $p\in Supp^{s}(U)$. So, for each $p\in Supp^{s}(U)$, there is $t_{p}\in
R\backslash p$ such that $t_{p}m=0$. Put $J=\tsum\limits_{p\in
Supp^{s}(U)}Rt_{p}$. Then $Jm=0$.

Let $q\in V(J)$. We claim that $ann_{R}(K)\subseteq q$. Suppose on the
contrary that $ann_{R}(K)\not\subseteq q$. Since $M$ is faithful secondful,
there is a second submodule $S$ of $M$ such that $q=ann_{R}(S)$. Therefore $%
S\in U$ and $q\in Supp^{s}(U)$. This implies that $t_{q}\in J\subseteq q$.
This contradicts the fact that $t_{q}\in R\backslash q$. Thus%
\begin{equation*}
ann_{R}(K)\subseteq \sqrt{ann_{R}(K)}=\tbigcap\limits_{p\in
V(ann_{R}(K))}p\subseteq \tbigcap\limits_{p\in V(J)}p=\sqrt{J}
\end{equation*}%
Since $R$ is Noetherian, $ann_{R}(K)^{n}\subseteq J$ for some $n\in 
\mathbb{Z}
^{+}$. Hence $ann_{R}(K)^{n}m\subseteq Jm=0$. This shows that $m\in \Gamma
_{ann_{R}(K)}(N)$ and the result follows.
\end{proof}

\begin{lemma}
\label{Lemma6}Let $r,n\in 
\mathbb{Z}
^{+}$, $s_{1},...,s_{r}\in R$. Then $V^{s}((0:_{M}\tsum%
\limits_{i=1}^{r}Rs_{i}^{n}))=V^{s}((0:_{M}\tsum\limits_{i=1}^{r}Rs_{i}))$.
\end{lemma}

\begin{proof}
Let $S\in V^{s}((0:_{M}\tsum\limits_{i=1}^{r}Rs_{i}^{n}))$. Then $%
ann_{R}((0:_{M}\tsum\limits_{i=1}^{r}Rs_{i}^{n}))=ann_{R}(\cap
_{i=1}^{r}(0:_{M}Rs_{i}^{n}))\subseteq ann_{R}(S)$. For each $i\in
\{1,...,r\}$, we have $(Rs_{i})^{n}=Rs_{i}^{n}\subseteq
ann_{R}((0:_{M}Rs_{i}^{n}))\subseteq ann_{R}(\cap
_{i=1}^{r}(0:_{M}Rs_{i}^{n}))\subseteq ann_{R}(S)$. Since $ann_{R}(S)$ is a
prime ideal, we have $Rs_{i}\subseteq ann_{R}(S)$ and so $%
\tsum\limits_{i=1}^{r}Rs_{i}\subseteq ann_{R}(S)$. It follows that $%
S\subseteq (0:_{M}ann_{R}(S))\subseteq (0:_{M}\tsum\limits_{i=1}^{r}Rs_{i})$%
. This shows that $S\in V^{s\ast
}((0:_{M}\tsum\limits_{i=1}^{r}Rs_{i}))=V^{s}((0:_{M}\tsum%
\limits_{i=1}^{r}Rs_{i}))$ and hence $V^{s}((0:_{M}\tsum%
\limits_{i=1}^{r}Rs_{i}^{n}))\subseteq
V^{s}((0:_{M}\tsum\limits_{i=1}^{r}Rs_{i}))$.

For the other containment; $\tsum\limits_{i=1}^{r}Rs_{i}^{n}\subseteq
\tsum\limits_{i=1}^{r}Rs_{i}$ implies that $(0:_{M}\tsum%
\limits_{i=1}^{r}Rs_{i})\subseteq (0:_{M}\tsum\limits_{i=1}^{r}Rs_{i}^{n})$
and hence $V^{s}((0:_{M}\tsum\limits_{i=1}^{r}Rs_{i}))\subseteq
V^{s}((0:_{M}\tsum\limits_{i=1}^{r}Rs_{i}^{n}))$.
\end{proof}

\begin{theorem}
\label{Theorem5}Let $R$ be a Noetherian ring, $M$ be a faithful secondful $R$%
-module, $N$ be an $R$-module and $U=X^{s}\backslash V^{s}(K)$ where $K\leq
M $. Let $W$ be an open subset of $X^{s}$ such that $U\subseteq W$. Then $%
ker(\rho _{WU})=\Gamma _{ann_{R}(K)}(\mathcal{O}(N,M)(W))$.
\end{theorem}

\begin{proof}
By Lemma \ref{Lemma3}, $\rho _{WU}(\Gamma _{ann_{R}(K)}(\mathcal{O}%
(N,M)(W))\subseteq \Gamma _{ann_{R}(K)}(\mathcal{O}(N,M)(U))=0$. So $\Gamma
_{ann_{R}(K)}(\mathcal{O}(N,M)(W))\subseteq ker(\rho _{WU})$.

There exists $L\leq M$ such that $W=X^{s}\backslash V^{s}(L)$. Let $\beta
=(\beta _{p})_{p\in Supp^{s}(W)}\in ker(\rho _{WU})$. By Lemma \ref{Lemma2},
there exist $r\in 
\mathbb{Z}
^{+},$ $s_{1},...,s_{r}\in ann_{R}(L)$ and $m_{1},...,m_{r}\in N$ such that $%
W=\cup _{i=1}^{r}Y_{s_{i}}$ and for each $i=1,...,r$ and each $S\in
Y_{s_{i}} $, we have $\beta _{ann_{R}(S)}=\frac{m_{i}}{s_{i}}$. Since $\beta
\in ker(\rho _{WU})$, we have $\beta _{p}=0$ for all $p\in Supp^{s}(U)$. Fix 
$i\in \{1,...,r\}$. Set $U^{\prime }=U\cap Y_{s_{i}}$. Hence,

$U^{\prime }=(X^{s}\backslash V^{s}(K))\cap (X^{s}\backslash
V^{s}((0:_{M}s_{i}))=X^{s}\backslash (V^{s}(K)\cup
V^{s}((0:_{M}s_{i})))=X^{s}\backslash V^{s}(K+(0:_{M}s_{i}))$.

Then $\frac{m_{i}}{s_{i}}=0\in N_{p}$ for all $p\in Supp^{s}(U^{\prime })$.
This implies that $\tau _{N}^{U^{\prime }}(m_{i})=0$. By Theorem \ref%
{Theorem4}, there exists $h_{i}\in 
\mathbb{Z}
^{+}$ such that $ann_{R}(K+(0:_{M}s_{i}))^{h_{i}}m_{i}=0$. Let $h:=\max
\{h_{1},...,h_{r}\}$. Now, let $p\in Supp^{s}(W)$. There exists $i\in
\{1,...,r\}$ such that $p\in Supp^{s}(Y_{s_{i}})$. Let $d\in ann_{R}(K)^{h}$%
. Since $s_{i}^{h}\in ann_{R}((0:_{M}s_{i}))^{h}$, we have

$ds_{i}^{h}\in ann_{R}(K)^{h}ann_{R}((0:_{M}s_{i}))^{h}\subseteq
(ann_{R}(K)\cap ann_{R}((0:_{M}s_{i})))^{h}=ann_{R}(K+(0:_{M}s_{i}))^{h}$.

Therefore, we conclude that $d\beta _{p}=\frac{dm_{i}}{s_{i}}=\frac{%
ds_{i}^{h}m_{i}}{s_{i}^{h+1}}=0\in N_{p}$ for all $d\in ann_{R}(K)^{h}$.
This implies that $ann_{R}(K)^{h}\beta =0$ and so $\beta \in \Gamma
_{ann_{R}(K)}(\mathcal{O}(N,M)(W))$.
\end{proof}

\begin{theorem}
\label{Theorem7}Let $R$ be a Noetherian ring, $M$ be a faithful secondful $R$%
-module, $N$ be an $R$-module and $U=X^{s}\backslash V^{s}(K)$ where $K\leq
M $. Then the map $\tau _{N}^{U}:N\longrightarrow \mathcal{O}(N,M)(U)$ has
an $ann_{R}(K)$-torsion cokernel.
\end{theorem}

\begin{proof}
Let $\beta =(\beta _{p})_{p\in Supp^{s}(U)}\in \mathcal{O}(N,M)(U)$. By
Lemma \ref{Lemma2}, there exist $r\in 
\mathbb{Z}
^{+}$, $s_{1},...,s_{r}\in ann_{R}(K)$ and $m_{1},...,m_{r}\in N$ such that $%
U=\tbigcup\limits_{i=1}^{r}Y_{s_{i}}$ and for each $i=1,...,r$ and each $%
S\in Y_{s_{i}}$ we have $\beta _{ann_{R}(S)}=\frac{m_{i}}{s_{i}}$. Fix $i\in
\{1,...,r\}$. Then for each $S\in Y_{s_{i}}$, $s_{i}\beta _{ann_{R}(S)}=%
\frac{s_{i}m_{i}}{s_{i}}\in N_{ann_{R}(S)}$. This means that%
\begin{equation*}
\rho _{UY_{s_{i}}}(s_{i}\beta )=s_{i}\rho _{UY_{s_{i}}}(\beta )=\tau
_{N}^{Y_{s_{i}}}(m_{i})=\rho _{UY_{s_{i}}}(\tau _{N}^{U}(m_{i}))
\end{equation*}%
Thus, $s_{i}\beta -\tau _{N}^{U}(m_{i})\in ker(\rho _{UY_{s_{i}}})=\Gamma
_{ann_{R}((0:_{M}s_{i}))}(\mathcal{O}(N,M)(U))$ by Theorem \ref{Theorem5}.
Hence there exists $n_{i}\in 
\mathbb{Z}
^{+}$ such that $ann_{R}((0:_{M}s_{i}))^{n_{i}}(s_{i}\beta -\tau
_{N}^{U}(m_{i}))=0$. Then $s_{i}^{n_{i}}(s_{i}\beta -\tau _{N}^{U}(m_{i}))=0$%
. Define $n:=\max \{n_{1},...n_{r}\}+1$. Then for all $i=1,...,r$, we have $%
s_{i}^{n}\beta =s_{i}^{n-n_{i}-1}s_{i}^{n_{i}+1}\beta
=s_{i}^{n-n_{i}-1}s_{i}^{n_{i}}\tau _{N}^{U}(m_{i})$. It follows that $%
s_{i}^{n}\beta =s_{i}^{n-1}\tau _{N}^{U}(m_{i})=\tau
_{N}^{U}(s_{i}^{n-1}m_{i})\in \tau _{N}^{U}(N)$. By Lemma \ref{Lemma6}, we
have

$V^{s}((0:_{M}\tsum\limits_{i=1}^{r}Rs_{i}^{n}))=V^{s}((0:_{M}\tsum%
\limits_{i=1}^{r}Rs_{i}))$

$\ \ \ \ \ \ \ \ \ \ \ \ \ \ \ \ \ \ \ \ \ \ \ \ \ \ =V^{s\ast
}((0:_{M}\tsum\limits_{i=1}^{r}Rs_{i}))$

$\ \ \ \ \ \ \ \ \ \ \ \ \ \ \ \ \ \ \ \ \ \ \ \ \ \ =V^{s\ast }(\cap
_{i=1}^{r}(0:_{M}Rs_{i}))$

$\ \ \ \ \ \ \ \ \ \ \ \ \ \ \ \ \ \ \ \ \ \ \ \ \ \ =\cap
_{i=1}^{r}V^{s\ast }((0:_{M}Rs_{i}))$

$\ \ \ \ \ \ \ \ \ \ \ \ \ \ \ \ \ \ \ \ \ \ \ \ \ \ =\cap
_{i=1}^{r}V^{s}((0:_{M}Rs_{i})).$

It follows that{}

$X^{s}\backslash
V^{s}((0:_{M}\tsum\limits_{i=1}^{r}Rs_{i}^{n}))=X^{s}\backslash \cap
_{i=1}^{r}V^{s}((0:_{M}Rs_{i}))$

$\ \ \ \ \ \ \ \ \ \ \ \ \ \ \ \ \ \ \ \ \ \ \ \ \ \ \ \ \ \ \ =\cup
_{i=1}^{r}(X^{s}\backslash V^{s}((0:_{M}Rs_{i})))$

$\ \ \ \ \ \ \ \ \ \ \ \ \ \ \ \ \ \ \ \ \ \ \ \ \ \ \ \ \ \ \ =\cup
_{i=1}^{r}Y_{s_{i}}=U=X^{s}\backslash V^{s}(K)$.

This means that $V^{s}((0:_{M}\tsum\limits_{i=1}^{r}Rs_{i}^{n}))=V^{s}(K)$.
Since $M$ is faithful secondful, we get that $V(\tsum%
\limits_{i=1}^{r}Rs_{i}^{n})\subseteq V(ann_{R}(K))$. Hence $%
ann_{R}(K)\subseteq \sqrt{ann_{R}(K)}\subseteq \sqrt{\tsum%
\limits_{i=1}^{r}Rs_{i}^{n}}$. Since $R$ is Noetherian, there exists $h\in 
\mathbb{Z}
^{+}$ such that $ann_{R}(K)^{h}\subseteq \tsum\limits_{i=1}^{r}Rs_{i}^{n}$.
It follows that $ann_{R}(K)^{h}\beta \subseteq \left(
\tsum\limits_{i=1}^{r}Rs_{i}^{n}\right) \beta \subseteq \tau _{N}^{U}(N)$.
This completes the proof.
\end{proof}

Let $K$ be an $R$-module and $I$ be an ideal of $R$. Recall that the \textit{%
ideal transform of }$K$\textit{\ with respect to }$I$ is defined as $%
D_{I}(K):=\underset{\overrightarrow{n\in 
\mathbb{N}
}}{\lim }Hom_{R}(I^{n}$, $K)$ (cf. \cite{Broadman-Sharp}).

\begin{theorem}
\label{Theorem8}Let $R$ be a Noetherian ring, $M$ be a faithful secondful $R$%
-module, $N$ be an $R$-module and $U=X^{s}\backslash V^{s}(K)$ where $K\leq
M $. Then, there is a unique $R$-isomorphism%
\begin{equation*}
g_{_{K,N}}:\mathcal{O}(N,M)(U)\longrightarrow D_{ann_{R}(K)}(N):=\lim_{%
\overrightarrow{n\in 
\mathbb{N}
}}Hom_{R}(ann_{R}(K)^{n},N)
\end{equation*}%
such that the diagram%
\begin{equation*}
\begin{array}{ccc}
N & \overset{\tau _{N}^{U}}{\longrightarrow } & \mathcal{O}(N,M)(U) \\ 
& \searrow & \downarrow g_{_{K,N}} \\ 
&  & D_{ann_{R}(K)}(N)%
\end{array}%
\end{equation*}%
commutes.
\end{theorem}

\begin{proof}
By Theorems \ref{Theorem4} and \ref{Theorem7}, both the kernel and cokernel
of $\tau _{N}^{U}$ are $ann_{R}(K)$-torsion. Therefore, there is a unique $R$%
-homomorphism $g_{_{K,N}}:\mathcal{O}(N,M)(U)\longrightarrow
D_{ann_{R}(K)}(N)$ such that the given diagram commutes by \cite[Corollary
2.2.13-(ii)]{Broadman-Sharp}. By Lemma \ref{Lemma3}, $\Gamma _{ann_{R}(K)}(%
\mathcal{O}(N,M)(U))=0$. So, $g_{_{K,N}}$ is an isomorphism by \cite[%
Corollary 2.2.13-(iii)]{Broadman-Sharp}.
\end{proof}

\begin{corollary}
\label{annK-torsion-0}Let $R$ be a Noetherian ring, $M$ be a faithful
secondful $R$-module, $N$ be an $R$-module and $U=X^{s}\backslash V^{s}(K)$
where $K\leq M$. Then the following hold.

$(1)$ $\mathcal{O}(\Gamma _{ann_{R}(K)}(N),M)(U)=0$

$(2)$ $\mathcal{O}(N,M)(U)\simeq \mathcal{O}(N/\Gamma _{ann_{R}(K)}(N),M)(U)$

$(3)$ $\mathcal{O}(N,M)(U)\simeq \mathcal{O}(\mathcal{O}(N,M)(U),M)(U)$

$(4)$ If $N$ is an $ann_{R}(K)$-torsion $R$-module, then $\mathcal{O}%
(N,M)(U)=0$.
\end{corollary}

\begin{proof}
Parts $(1)$, $(2)$ and $(3)$ follow from Theorem \ref{Theorem8} and \cite[%
Corollary 2.2.8]{Broadman-Sharp}. Part $(4)$ is an immediate consequence of
part $(1)$.
\end{proof}

\begin{example}
Consider the $%
\mathbb{Z}
$-modules $N=%
\mathbb{Z}
/9%
\mathbb{Z}
$ and $M=(\oplus _{p}%
\mathbb{Z}
_{p})\oplus 
\mathbb{Q}
$ where $p$ runs over all distinct prime numbers. Then $M$ is a faithful
secondful $%
\mathbb{Z}
$-module. Let $K=(0\oplus 
\mathbb{Z}
_{3}\oplus 0...)\oplus 0$ and $U=Spec^{s}(M)\backslash V^{s}(K)$. Then $%
ann_{R}(K)=3%
\mathbb{Z}
$ and $N$ is a $3%
\mathbb{Z}
$-torsion $%
\mathbb{Z}
$-module. By Corollary \ref{annK-torsion-0}-$(4)$, $\mathcal{O}(%
\mathbb{Z}
/9%
\mathbb{Z}
$, $(\oplus _{p}%
\mathbb{Z}
_{p})\oplus 
\mathbb{Q}
)(U)=0$.
\end{example}

\begin{corollary}
Let $R$ be a principal ideal domain, $M$ be a faithful secondful $R$-module, 
$N$ be an $R$-module and $U=X^{s}\backslash V^{s}(K)$. Then there exists $%
a\in R$ such that $\mathcal{O}(N,M)(U)\simeq N_{a}$, where $N_{a}$ is the
localization of $N$ with respect to the multiplicative set $\{a^{n}:n\in 
\mathbb{N}
\}$.
\end{corollary}

\begin{proof}
Since $R$ is a a principal ideal domain, there is an element $a\in R$ such
that $ann_{R}(K)=Ra$. By Theorem \ref{Theorem8} and \cite[Theorem 2.2.16]%
{Broadman-Sharp}, we have $\mathcal{O}(N,M)(U)\simeq D_{ann_{R}(K)}(N)\simeq
N_{a}$.
\end{proof}

\begin{theorem}
\label{Theorem11}Let $M$ be a faithful secondful $R$-module and $N$ be any $%
R $-module. For any element $f\in R$, the module $\mathcal{O}(N,M)(Y_{f})$
is isomorphic to the localized module $N_{f}$. In particular, $\mathcal{O}%
(N,M)(X^{s})\simeq N$.
\end{theorem}

\begin{proof}
We define the map $\phi :N_{f}\longrightarrow \mathcal{O}(N,M)(Y_{f})$ by $%
\phi (\frac{a}{f^{m}})=(\frac{a}{f^{m}})_{p\in Supp^{s}(Y_{f})}$. We claim
that $\phi $ is an isomorphism. First, we show that $\phi $ is injective.
Let $\phi (\frac{a}{f^{n}})=\phi (\frac{b}{f^{m}})$. Then for every $S\in
Y_{f}$, $\frac{a}{f^{n}}=\frac{b}{f^{m}}$ in $N_{p}$ where $p=ann_{R}(S)$.
Thus there exists $h\in R\backslash p$ such that $h(f^{m}a-f^{n}b)=0$ in $N$%
. Let $I=(0:_{R}f^{m}a-f^{n}b)$. Then $h\in I$ and $h\not\in p$, so $%
I\not\subseteq p$. This holds for any $S\in Y_{f}$, so we deduce that $%
Supp^{s}(Y_{f})\subseteq Spec(R)\backslash V(I)$. Since $M$ is faithful
secondful, $D_{f}=Supp^{s}(Y_{f})\subseteq Spec(R)\backslash V(I)$ and we
get that $V(I)\subseteq V(Rf)$. This implies that $Rf\subseteq \sqrt{Rf}%
\subseteq \sqrt{I}$. Therefore $f^{l}\in I$ for some $l\in 
\mathbb{Z}
^{+}$. Now we have $f^{l}(f^{m}a-f^{n}b)=0$ which shows that $\frac{a}{f^{n}}%
=\frac{b}{f^{m}}$ in $N_{f}$. Thus $\phi $ is injective.

Let $\beta =(\beta _{p})_{p\in Supp^{s}(Y_{f})}\in \mathcal{O}(N,M)(Y_{f})$.
Then we can cover $Y_{f}$ with the open subsets $V_{i}$ on which $\beta
_{ann_{R}(S)}$ is represented by $\frac{a_{i}}{g_{i}}$ with $g_{i}\not\in
ann_{R}(S)$ for all $S\in V_{i}$, in other words $V_{i}\subseteq Y_{g_{i}}$.
Since the open sets of the form $Y_{r}$ $(r\in R)$ form a base for the dual
Zariski topology on $X^{s}$, we may assume that $V_{i}=Y_{h_{i}}$ for some $%
h_{i}\in R$. Since $Y_{h_{i}}\subseteq Y_{g_{i}}$, $D_{h_{i}}=\psi
^{s}(Y_{h_{i}})\subseteq \psi ^{s}(Y_{g_{i}})=D_{g_{i}}$ by \cite[%
Proposition 4.1]{ATF1}. This implies that $V(Rg_{i})\subseteq V(Rh_{i})$ and
so $Rh_{i}\subseteq \sqrt{Rh_{i}}\subseteq \sqrt{Rg_{i}}$. Thus $%
h_{i}^{s}\in Rg_{i}$ for some $s\in 
\mathbb{Z}
^{+}$. So $h_{i}^{s}=cg_{i}$ for some $c\in R$ and $\frac{a_{i}}{g_{i}}=%
\frac{ca_{i}}{cg_{i}}=\frac{ca_{i}}{h_{i}^{s}}$. We see that $\beta
_{ann_{R}(S)}$ is represented by $\frac{b_{i}}{k_{i}}$ $(b_{i}=ca_{i},$ $%
k_{i}=h_{i}^{s})$ on $Y_{k_{i}}$ and (since $Y_{h_{i}}=Y_{h_{i}^{s}}$) the $%
Y_{k_{i}}$ cover $Y_{f}$. The open cover $Y_{f}=\cup Y_{k_{i}}$ has a finite
subcover by \cite[Theorem 4-(2)]{dual-zariski}. Suppose that $Y_{f}\subseteq
Y_{k_{1}}\cup ...\cup Y_{k_{n}}.$For $1\leq i,j\leq n,$ $\frac{b_{i}}{k_{i}}$
and $\frac{b_{j}}{k_{j}}$ both represent $\beta _{ann_{R}(S)}$ on $%
Y_{k_{i}}\cap Y_{k_{j}}$. By \cite[Corollary 4.2]{ATF1}, $Y_{k_{i}}\cap
Y_{k_{j}}=Y_{k_{i}k_{j}}$, and by the injectivity of $\phi $, we get that $%
\frac{b_{i}}{k_{i}}=\frac{b_{j}}{k_{j}}$ in $N_{k_{i}k_{j}}$. Hence $%
(k_{i}k_{j})^{n_{ij}}(k_{j}b_{i}-k_{i}b_{j})=0$ for some $n_{ij}\in 
\mathbb{Z}
^{+}$. Let $m=\max \{n_{ij}:1\leq i,j\leq n\}$. Then $%
k_{j}^{m+1}(k_{i}^{m}b_{i})-k_{i}^{m+1}(k_{j}^{m}b_{j})=0$. By replacing
each $k_{i}$ by $k_{i}^{m+1}$ and $b_{i}$ by $k_{i}^{m}b_{i}$, we still see
that $\beta _{ann_{R}(S)}$ is represented on $Y_{k_{i}}$ by $\frac{b_{i}}{%
k_{i}}$ and furthermore we have $k_{j}b_{i}=k_{i}b_{j}$ for all $i,j$. Since 
$Y_{f}\subseteq Y_{k_{1}}\cup ...\cup Y_{k_{n}}$, by \cite[Proposition 4.1]%
{ATF1}, we have

$D_{f}=\psi ^{s}(Y_{f})\subseteq \cup _{i=1}^{n}\psi ^{s}(Y_{k_{i}})=\cup
_{i=1}^{n}D_{k_{i}}$, i.e., $Spec(R)\backslash V(Rf)\subseteq
Spec(R)\backslash \cap _{i=1}^{n}V(Rk_{i})$. This implies that $\cap
_{i=1}^{n}V(Rk_{i})=V(\sum_{i=1}^{n}Rk_{i})\subseteq V(Rf)$ and hence $%
Rf\subseteq \sqrt{Rf}\subseteq \sqrt{\sum_{i=1}^{n}Rk_{i}}$. So, there are $%
c_{1},...,c_{n}\in R$ and $t\in 
\mathbb{Z}
^{+}$ such that $f^{t}=\sum_{i=1}^{n}c_{i}k_{i}$. Let $a=%
\sum_{i=1}^{n}c_{i}b_{i}$. Then for each $j$, we have $k_{j}a=%
\sum_{i=1}^{n}c_{i}k_{j}b_{i}=\sum_{i=1}^{n}c_{i}k_{i}b_{j}=b_{j}f^{t}$.
This implies that $\frac{a}{f^{t}}=\frac{b_{j}}{k_{j}}$ on $Y_{k_{j}}$.
Therefore, $\phi (\frac{a}{f^{t}})=(\beta _{p})_{p\in Supp^{s}(Y_{f})}$.
\end{proof}

\begin{proposition}
\label{O(K,M)(U)izoO(L,M)(U)}Let $K$, $L$ be $R$-modules and $\varphi
:K\longrightarrow L$ be an $R$-homomorphism. Then $\varphi $ induces a
morphism of sheaves $\overline{\varphi }:\mathcal{O}(K,M)\longrightarrow 
\mathcal{O}(L,M)$. If $\varphi $ is an isomorphism of $R$-modules, then $%
\overline{\varphi }$ is an isomorphism of sheaves.
\end{proposition}

\begin{proof}
Let $U$ be an open subset of $X^{s}$ and $\beta =(\frac{a_{p}}{f_{p}})_{p\in
Supp^{s}(U)}\in \mathcal{O}(K,M)(U)$. We show that $\left( \frac{\varphi
(a_{p})}{f_{p}}\right) _{p\in Supp^{s}(U)}\in \mathcal{O}(L,M)(U)$. For each 
$Q\in U$, there is an open neighborhood $W$ of $Q$ with $Q\in W\subseteq U$
and there exist elements $t\in R$, $m\in K$ such that for every $S\in W$, we
have $t\not\in p:=ann_{R}(S)$ and $\frac{a_{p}}{f_{p}}=\frac{m}{t}\in K_{p}$%
. So there exists $s_{p}\in R\backslash p$ such that $s_{p}(ta_{p}-f_{p}m)=0$%
. It follows that $s_{p}\left( t\varphi (a_{p})-f_{p}\varphi (m)\right) =0$.
This means that $\frac{\varphi (a_{p})}{f_{p}}=\frac{\varphi (m)}{t}$ where $%
\varphi (m)\in L$ and $t\not\in p:=ann_{R}(S)$ for every $S\in W$. This
shows that $\left( \frac{\varphi (a_{p})}{f_{p}}\right) _{p\in
Supp^{s}(U)}\in \mathcal{O}(L,M)(U)$. Thus the map $\overline{\varphi }(U):%
\mathcal{O}(K,M)(U)\longrightarrow \mathcal{O}(L,M)(U)$ defined by 
\begin{equation*}
\overline{\varphi }(U)\left( (\frac{a_{p}}{f_{p}})_{p\in Supp^{s}(U)}\right)
=\left( \frac{\varphi (a_{p})}{f_{p}}\right) _{p\in Supp^{s}(U)}.
\end{equation*}%
is well-defined. Clearly, $\overline{\varphi }(U)$ is an $R$-homomorphism.
Since 
\begin{eqnarray*}
\rho _{UV}^{\prime }\overline{\varphi }(U)\left( (\frac{a_{p}}{f_{p}})_{p\in
Supp^{s}(U)}\right) &=&\rho _{UV}^{\prime }\left( (\frac{\varphi \left(
a_{p}\right) }{f_{p}})_{p\in SuppU}\right) \\
&=&(\frac{\varphi \left( n_{p}\right) }{f_{p}})_{p\in SuppV} \\
&=&\overline{\varphi }(V)\rho _{UV}\left( (\frac{a_{p}}{f_{p}})_{p\in
Supp^{s}(U)}\right) ,
\end{eqnarray*}%
the following diagram is commutative.%
\begin{equation*}
\begin{array}{ccc}
\mathcal{O}(K,M)(U) & \underrightarrow{\overline{\varphi }(U)} & \mathcal{O}%
(L,M)(U) \\ 
\downarrow \rho _{UV} &  & \downarrow \rho _{UV}^{\prime } \\ 
\mathcal{O}(K,M)(V) & \underrightarrow{\overline{\varphi }(V)} & \mathcal{O}%
(L,M)(V)%
\end{array}%
\end{equation*}%
This shows that $\overline{\varphi }:\mathcal{O}(K,M)\longrightarrow 
\mathcal{O}(L,M)$ is a morphism of sheaves.

Now suppose that $\varphi $ is an isomorphism. We show that $\overline{%
\varphi }(U)$ is injective. Let $\overline{\varphi }(U)\left( (\frac{a_{p}}{%
f_{p}})_{p\in Supp^{s}(U)}\right) =\left( \frac{\varphi (a_{p})}{f_{p}}%
\right) _{p\in Supp^{s}(U)}=0$. Then $\frac{\varphi (a_{p})}{f_{p}}=0$ for
every $p\in Supp^{s}(U)$. There exists $t_{p}\in R\backslash p$ such that $%
t_{p}\varphi (a_{p})=\varphi (t_{p}a_{p})=0$. Since $\varphi $ is injective, 
$t_{p}a_{p}=0$ for every $p\in Supp^{s}(U)$. It follows that $\frac{a_{p}}{%
f_{p}}=\frac{t_{p}a_{p}}{t_{p}f_{p}}=0$ for every $p\in Supp^{s}(U)$. This
shows that $(\frac{a_{p}}{f_{p}})_{p\in Supp^{s}(U)}=0$ and so $\overline{%
\varphi }(U)$ is injective for every open subset $U$ of $X^{s}$.

Now we show that $\overline{\varphi }(U)$ is surjective. Let $\left( \frac{%
b_{p}}{t_{p}}\right) _{p\in Supp^{s}(U)}\in \mathcal{O}(L,M)(U)$. There
exists $a_{p}\in K$ such that $\varphi (a_{p})=b_{p}$ for each $p\in
Supp^{s}(U)$. We show that $\left( \frac{a_{p}}{t_{p}}\right) _{p\in
Supp^{s}(U)}\in \mathcal{O}(K,M)(U)$. For each $Q\in U$, there is an open
neighborhood $W$ of $Q$ with $Q\in W\subseteq U$ and there exist elements $%
t\in R$, $b\in L$ such that for every $S\in W$, we have $t\not\in
ann_{R}(S)=p$ and $\frac{b_{p}}{t_{p}}=\frac{\varphi (a_{p})}{t_{p}}=\frac{b%
}{t}$. There exists $a\in K$ such that $b=\varphi (a)$. So $\frac{\varphi (a)%
}{t}=\frac{\varphi (a_{p})}{t_{p}}$, where $t\not\in p=ann_{R}(S)$ for every 
$S\in W$. There exists $v_{p}\in R\backslash p$ such that $%
v_{p}(t_{p}\varphi (a)-t\varphi (a_{p}))=0$. It follows that $\varphi
(v_{p}t_{p}a-v_{p}ta_{p})=0$. Since $\varphi $ is injective $%
v_{p}(t_{p}a-ta_{p})=0$ for $v_{p}\in R\backslash p$. This means that $\frac{%
a_{p}}{t_{p}}=\frac{a}{t}$ where $t\not\in p=ann_{R}(S)$ for every $S\in W$.
This shows that $\left( \frac{a_{p}}{t_{p}}\right) _{p\in Supp^{s}(U)}\in 
\mathcal{O}(K,M)(U)$ and $\varphi \left( \left( \frac{a_{p}}{t_{p}}\right)
_{p\in Supp^{s}(U)}\right) =\left( \frac{b_{p}}{t_{p}}\right) _{p\in
Supp^{s}(U)}$. Thus $\overline{\varphi }(U)$ is surjective for every open
subset $U$ of $X^{s}$. Consequently, $\overline{\varphi }$ is an isomorphism
of sheaves.
\end{proof}

\begin{theorem}
\label{Free-O(N,M)(X)}Let $R$ be a Noetherian ring, $M$ be a faithful
secondful $R$-module and $N$ be an $R$-module. Then the following hold.

$(1)$ If $N$ is a free $R$-module, then $\mathcal{O}(N,M)(X^{s})$ is a free $%
R$-module.

$(2)$ If $N$ is a projective $R$-module, then $\mathcal{O}(N,M)(X^{s})$ is a
projective $R$-module.

$(3)$ If $N$ is a flat $R$-module, then $\mathcal{O}(N,M)(X^{s})$ is a flat $%
R$-module.
\end{theorem}

\begin{proof}
$(1)$ We can write $X^{s}=X^{s}\backslash V^{s}(0)$. Since $N$ is a free $R$%
-module, $N$ is isomorphic to a direct sum of some copies of $R$, say $%
N\simeq \oplus _{i\in \Lambda }R$ for an index set $\Lambda $. By
Proposition \ref{O(K,M)(U)izoO(L,M)(U)}, $\mathcal{O}(N,M)(X^{s})\simeq 
\mathcal{O}(\oplus _{i\in \Lambda }R,M)(X^{s})$. By Theorem \ref{Theorem8}, $%
\mathcal{O}(N,M)(X^{s})\simeq D_{R}(N)$ and $\mathcal{O}(\oplus _{i\in
\Lambda }R,M)(X^{s})\simeq D_{R}(\oplus _{i\in \Lambda }R)$. $D_{R}$
commutes with direct sums by \cite[3.4.11. Corollary]{Broadman-Sharp}. By
using this fact, Theorem \ref{Theorem8} and Theorem \ref{Theorem11}, we get
that $\mathcal{O}(N,M)(X^{s})\simeq D_{R}(\oplus _{i\in \Lambda }R)\simeq
\oplus _{i\in \Lambda }D_{R}(R)\simeq \oplus _{i\in \Lambda }\mathcal{O}%
(R,M)(X^{s})\simeq \oplus _{i\in \Lambda }R$. This shows that $\mathcal{O}%
(N,M)(X^{s})$ is a free $R$-module.

$(2)$ Since $N$ is a projective $R$-module, there is a free $R$-module $F$
and a submodule $L$ of $F$ such that $F\simeq N\oplus L$. By using
Proposition \ref{O(K,M)(U)izoO(L,M)(U)} and \cite[3.4.11. Corollary]%
{Broadman-Sharp}, we get that $\mathcal{O}(F,M)(X^{s})\simeq \mathcal{O}%
(N\oplus L,M)(X^{s})\simeq \mathcal{O}(N,M)(X^{s})\oplus \mathcal{O}%
(L,M)(X^{s})$. By part $(1)$, $\mathcal{O}(F,M)(X^{s})$ is a free $R$%
-module. $\mathcal{O}(N,M)(X^{s})$ is a projective $R$-module as it is a
direct summand of the free $R$-module $\mathcal{O}(F,M)(X^{s})$.

$(3)$ Since every flat $R$-module is a direct limit of projective $R$%
-modules, $N\simeq \underset{\overrightarrow{i\in \Lambda }}{\lim }P_{i}$
for some projective $R$-modules $P_{i}$ and a directed set $\Lambda $. By
Proposition \ref{O(K,M)(U)izoO(L,M)(U)} and Theorem \ref{Theorem8}, $%
\mathcal{O}(N,M)(X^{s})\simeq \mathcal{O}(\underset{\overrightarrow{i\in
\Lambda }}{\lim }P_{i},M)(X^{s})\simeq D_{R}(\underset{\overrightarrow{i\in
\Lambda }}{\lim }P_{i})$. $D_{R}$ commutes with direct limits by \cite[%
3.4.11. Corollary]{Broadman-Sharp}. By using this fact, Theorem \ref%
{Theorem8} and Theorem \ref{Theorem11}, we get that $\mathcal{O}%
(N,M)(X^{s})\simeq D_{R}(\underset{\overrightarrow{i\in \Lambda }}{\lim }%
P_{i})\simeq \underset{\overrightarrow{i\in \Lambda }}{\lim }%
D_{R}(P_{i})\simeq \underset{\overrightarrow{i\in \Lambda }}{\lim }\mathcal{O%
}(P_{i},M)(X^{s})$. By part $(2)$, $\mathcal{O}(P_{i},M)(X^{s})$ is a
projective and hence a flat $R$-module for each $i\in \Lambda $. Since a
direct limit of flat modules is flat, $\mathcal{O}(N,M)(X^{s})$ is a flat $R$%
-module.
\end{proof}

\section{A Scheme Structure On The Second Spectrum Of A Module}

Recall that an \textit{affine scheme} is a locally ringed space which is
isomorphic to the spectrum of some ring. A \textit{scheme }is a locally
ringed space $(X,\mathcal{O}_{X})$ in which every point has an open
neighborhood $U$ such that the topological space $U$, together with the
restricted sheaf $\mathcal{O}_{X\mid U}$ is an affine scheme. A scheme $(X,%
\mathcal{O}_{X})$ is called \textit{locally Noetherian} if it can be covered
by open affine subsets of $Spec(A_{i})$ where each $A_{i}$ is a Noetherian
ring. The scheme $(X,\mathcal{O}_{X})$ is called \textit{Noetherian }if it
is locally Noetherian and quasi-compact (cf. \cite{Hartshone}).

A topological space $X$ is said to be a $T_{0}$-space or a Kolmogorov space
if for every pair of distinct points $x,y\in X$, there exists open
neighbourhoods $U$ of $x$ and $V$ of $y$ such that either $x\notin V$ or $%
y\notin U$. The following proposition, from \cite{ATF1}, gives some
conditions for the dual Zariski topology of an $R$-module to be a $T_{0}$%
-space.

\begin{proposition}
\label{Lemma 2.1}\cite[Theorem 6.3]{ATF1} The following statements are
equivalent for an $R$-module $M$.

$(1)$ The natural map $\psi ^{s}:Spec^{s}(M)\longrightarrow
Spec(R/ann_{R}(M))$ is injective.

$(2)$ For any $S_{1}$, $S_{2}\in $Spec$^{s}(M)$, if $%
V^{s}(S_{1})=V^{s}(S_{2})$ then $S_{1}=S_{2}$.

$(3)$ $\mid Spec_{p}^{s}(M)\mid \leq 1$ for every $p\in Spec(R)$, where $%
Spec_{p}^{s}(M)$ is the set of all $p$-second submodules of $M$.

$(4)$ $(Spec^{s}(M),\tau ^{s})$ is a $T_{0}$-space.
\end{proposition}

\begin{theorem}
\label{Scheme}Let $M$ be a faithful secondful $R$-module such that $X^{s}$
is a $T_{0}$-space. Then $(X^{s},\mathcal{O}(R,M))$ is a scheme. Moreover,
if $R$ is Noetherian, then $(X^{s},\mathcal{O}(R,M))$ is a Noetherian scheme.
\end{theorem}

\begin{proof}
Let $g\in R$. Since the natural map $\psi
_{M}^{s}:Spec^{s}(M)\longrightarrow Spec(R)$ is continuous by \cite[%
Proposition 3.6]{ATF1}, the restriction map $\psi
_{M|Y_{g}}^{s}:Y_{g}\longrightarrow \psi _{M}^{s}(Y_{g})$ is also
continuous. Since $Y_{g}$ is also a $T_{0}$-space, $\psi _{M|Y_{g}}^{s}$ is
a bijection. Let $E$ be closed subset of $Y_{g}$. Then $E=Y_{g}\cap V^{s}(N)$
for some $N\leq M$. Hence $\psi _{M}^{s}(E)=\psi _{M}^{s}(Y_{g}\cap
V^{s}(N))=\psi _{M}^{s}(Y_{g})\cap \psi _{M}^{s}(V^{s}(N))=\psi
_{M}^{s}(Y_{g})\cap V(ann_{R}(N))$ is a closed subset of $\psi
_{M}^{s}(Y_{g})$. Therefore, $\psi _{M|Y_{g}}^{s}$ is a homeomorphism.

Since the sets of the form $Y_{g}$ $(g\in R)$ form a base for the dual
Zariski topology, $X^{s}$ can be written as $X^{s}=\cup _{i\in I}Y_{g_{i}}$
for some $g_{i}\in R$. Since $M$ is faithful secondful and $X^{s}$ is a $%
T_{0}$-space, we have $Y_{g_{i}}\simeq \psi
_{M}^{s}(Y_{g_{i}})=D_{g_{i}}\simeq Spec(R_{g_{i}})$ for each $i\in I$. By
Theorem \ref{Theorem11}, $Y_{g_{i}}$ is an affine scheme for each $i\in I$.
This implies that $(X^{s},\mathcal{O}(R,M))$ is a scheme.

For the last statement, we note that since $R$ is Noetherian so is $%
R_{g_{i}} $ for each $i\in I$. Hence $(X^{s},\mathcal{O}(R,M))$ is a locally
Noetherian scheme. By \cite[Theorem 4]{dual-zariski}, $X^{s}$ is
quasi-compact. Therefore, $(X^{s},\mathcal{O}(R,M))$ is a Noetherian scheme.
\end{proof}

\begin{theorem}
\label{morphism1}Let $M$ and $N$ be $R$-modules and $\phi :M\longrightarrow
N $ be a monomorphism. Then $\phi $ induces a morphism of locally ringed
spaces%
\begin{equation*}
(f,f^{\#}):(Spec^{s}(N),\mathcal{O}(R,N))\longrightarrow (Spec^{s}(M),%
\mathcal{O}(R,M)).
\end{equation*}
\end{theorem}

\begin{proof}
By \cite[Proposition 3.14]{ATF1}, the map $f:Spec^{s}(M)\longrightarrow
Spec^{s}(N)$ which is defined by $f(S)=\phi (S)$ for every $S\in
Spec^{s}(M), $ is continuous. Let $U$ be an open subset of $Spec^{s}(N)$ and 
$\beta \in \mathcal{O}(R,N)(U)$. Suppose $S\in f^{-1}(U)$. Then $f(S)=\phi
(S)\in U$. There exists an open neighborhood $W$ of $\phi (S)$ with $\phi
(S)\in W\subseteq U$ such that for each $Q\in W$, $g\not\in q=ann_{R}(Q)$
and $\beta _{q}=\frac{a}{g}$ in $R_{q}$. Since $\phi (S)=f(S)\in W$, $S\in
f^{-1}(W)\subseteq f^{-1}(U)$. As $f$ is continuous, $f^{-1}(W)$ is an open
neighborhood of $S$. We claim that for each $Q^{\prime }\in f^{-1}(W)$, $%
g\not\in ann_{R}(Q^{\prime })$. Suppose on the contrary that $g\in
ann_{R}(Q^{\prime })$ for some $Q^{\prime }\in f^{-1}(W)$. Then $f(Q^{\prime
})=\phi (Q^{\prime })\in W$. Since $\phi $ is a monomorphism, $%
ann_{R}(Q^{\prime })=ann_{R}(\phi (Q^{\prime }))$. So $g\in ann_{R}(\phi
(Q^{\prime }))$ for $\phi (Q^{\prime })\in W$, a contradiction. Therefore
for every open subset $U$ of $Spec^{s}(N)$, we can define the map%
\begin{equation*}
f^{\#}(U):\mathcal{O}(R,N)(U)\longrightarrow \mathcal{O}(R,M)(f^{-1}(U))
\end{equation*}%
as follows. For every $\beta \in \mathcal{O}(R,N)(U)$, $f^{\#}(U)(\beta )\in 
\mathcal{O}(R,M)(f^{-1}(U))$ is defined by $f^{\#}(U)(\left( \beta
_{p}\right) _{p\in Supp^{s}(U)})=(\beta _{ann_{R}(f(S)})_{S\in f^{-1}(U)}$.
As we mentioned above $f^{\#}(U)$ is a well-defined map an clearly it is a
ring homomorphism. Now we show that $f^{\#}$ is a locally ringed morphism.
Assume that $U$ and $V$ are open subsets of $Spec^{s}(N)$ with $V\subseteq U$
and $\beta =\left( \beta _{p}\right) _{p\in Supp^{s}(U)}\in \mathcal{O}%
(R,N)(U)$. Consider the diagram%
\begin{equation*}
\begin{array}{cccc}
\mathcal{O}(R,N)(U) & \underrightarrow{f^{\#}(U)} & \mathcal{O}%
(R,M)(f^{-1}(U)) &  \\ 
\downarrow \rho _{UV}^{\prime } &  & \downarrow \rho _{f^{-1}(U)f^{-1}(V)} & 
\\ 
\mathcal{O}(R,N)(V) & \underrightarrow{f^{\#}(V)} & \mathcal{O}%
(R,M)(f^{-1}(V)). &  \\ 
&  &  & 
\end{array}%
\end{equation*}%
Then%
\begin{eqnarray*}
\rho _{f^{-1}(U)f^{-1}(V)}f^{\#}(U)(\left( \beta _{p}\right) _{p\in
Supp^{s}(U)}) &=&\rho _{f^{-1}(U)f^{-1}(V)}((\beta _{ann_{R}(f(S)})_{S\in
f^{-1}(U)}) \\
&=&(\beta _{ann_{R}(f(S)})_{S\in f^{-1}(V)}=f^{\#}(V)((\beta _{p})_{p\in
Supp^{s}(V)}) \\
&=&f^{\#}(V)\rho _{UV}^{\prime }(\left( \beta _{p}\right) _{p\in
Supp^{s}(U)}).
\end{eqnarray*}

Therefore we get that $\rho _{f^{-1}(U)f^{-1}(V)}f^{\#}(U)=f^{\#}(V)\rho
_{UV}^{\prime }$. Thus the above diagram is commutative. This shows that $%
f^{\#}:\mathcal{O}(R,N)\longrightarrow f_{\ast }(\mathcal{O}(R,M))$ is a
morphism of sheaves. By Theorem \ref{Stalk}, the map on the stalks, $%
f_{S}^{\#}:\mathcal{O}(R,N)_{f(S)}\longrightarrow \mathcal{O}(R,M)_{S}$ is
clearly the map of local rings $R_{ann_{R}(f(S))}\longrightarrow
R_{ann_{R}(S)}$ which maps $\frac{r}{s}\in R_{ann_{R}(f(S))}$ to $\frac{r}{s}
$ again. This implies that $(f,f^{\#}):(Spec^{s}(N),\mathcal{O}%
(R,N))\longrightarrow (Spec^{s}(M),\mathcal{O}(R,M))$ is a morphism of
locally ringed spaces.
\end{proof}

\begin{theorem}
\label{Theorem13}Let $M$, $A$ be $R$-modules, $\phi :R\longrightarrow S$ be
a ring homomorphism, let $N,B$ be $S$-modules and $M$ be a secondful $R$%
-module such that $Spec^{s}(M)$ is a $T_{0}$-space and $ann_{R}(M)\subseteq
ann_{R}(N)$. If $\delta :A\longrightarrow B$ is an $R$-homomorphism, then $%
\phi $ induces a morphism of sheaves%
\begin{equation*}
h^{\#}:\mathcal{O}(A,M)\longrightarrow h_{\ast }(\mathcal{O}(B,N))
\end{equation*}
\end{theorem}

\begin{proof}
Since $ann_{R}(M)\subseteq ann_{R}(N)$, $\phi $ induces the homomorphism%
\begin{equation*}
\Theta :R/ann_{R}(M)\longrightarrow S/ann_{S}(N),\Theta (r+ann_{R}(M))=\phi
(r)+ann_{S}(N).
\end{equation*}

It is well-known that the maps

$f:Spec(S)\longrightarrow Spec(R)$ defined by $f(p)=\phi ^{-1}(p)$ and

$d:Spec(S/ann_{S}(N))\longrightarrow Spec(R/ann_{R}(M))$ defined by $d(%
\overline{p})=\Theta ^{-1}(\overline{p})$ and

$\psi _{N}^{s}:Spec^{s}(N)\longrightarrow Spec(S/ann_{S}(N))$ defined by $%
\psi _{N}^{s}(Q)=ann_{S}(Q)/ann_{S}(N)$ for each $Q\in Spec^{s}(N)$ are
continuous. Also, $\psi _{M}^{s}:Spec^{s}(M)\longrightarrow
Spec(R/ann_{R}(M))$ is a homeomorphism by \cite[Theorem 6.3]{ATF1}.
Therefore the map $h:Spec^{s}(N)\longrightarrow Spec^{s}(M)$ defined by $%
h(Q)=(\psi _{M}^{s})^{-1}d\psi _{N}^{s}(Q)$ is continuous. Also, for each $%
Q\in Spec^{s}(N)$, we get an $R$-homomorphism%
\begin{equation*}
\begin{array}{cccc}
\phi _{ann_{S}(Q)}: & A_{f(ann_{S}(Q))} & \longrightarrow & B_{ann_{S}(Q)}
\\ 
& \frac{a}{s} & \longrightarrow & \frac{\delta (r)}{\phi (s)}%
\end{array}%
\end{equation*}%
Let $U$ be an open subset of $Spec^{s}(M)$ and $t=(t_{ann_{R}(P)})_{P\in
U}\in \mathcal{O}(A,M)(U)$. Suppose that $T\in h^{-1}(U)$. Then $h(T)\in U$
and there exists an open neighborhood $W$ of $h(T)$ with $h(T)\in W\subseteq
U$ and elements $r,g\in R$ such that for each $Q\in W$, we have $%
t_{ann_{R}(Q)}=\frac{a}{g}\in A_{ann_{R}(Q)}$ where $g\not\in ann_{R}(Q)$,
hence $g\not\in ann_{R}(h(T))$. By definition of $h$, $ann_{R}(h(K))=\phi
^{-1}(ann_{S}(K))$ for every $K\in h^{-1}(W)$. So $\phi (g)\not\in
ann_{S}(K) $ for $g\not\in ann_{R}(h(K))$. Thus $\phi _{ann_{S}(T)}(\frac{a}{%
g})=\frac{\delta (a)}{\phi (g)}$ define a section $\mathcal{O}%
(B,N)(h^{-1}(U))$. We can define%
\begin{equation*}
h^{\#}(U):\mathcal{O}(A,M)(U)\longrightarrow h_{\ast }(\mathcal{O}(B,N)(U))=%
\mathcal{O}(B,N)(h^{-1}(U))
\end{equation*}%
by $h^{\#}(U)((t_{ann_{R}(P)})_{P\in U})=\left( \phi
_{ann_{S}(T)}(t_{ann_{R}(h(T))})\right) _{T\in h^{-1}(U)}$ for each $%
(t_{ann_{R}(P)})_{P\in U}\in \mathcal{O}(R,M)(U)$. Assume that $V\subseteq U$%
. Consider the diagram%
\begin{equation*}
\begin{array}{ccc}
\mathcal{O}(A,M)(U) & \underrightarrow{h^{\#}(U)} & \mathcal{O}%
(B,N)(h^{-1}(U)) \\ 
\downarrow \rho _{UV} &  & \downarrow \rho _{h^{-1}(U)h^{-1}(V)}^{\prime }
\\ 
\mathcal{O}(A,M)(V) & \underrightarrow{h^{\#}(V)} & \mathcal{O}%
(B,N)(h^{-1}(V))%
\end{array}%
\end{equation*}%
We see that%
\begin{eqnarray*}
\rho _{h^{-1}(U)h^{-1}(V)}^{\prime }h^{\#}(U)((t_{ann_{R}(P)})_{P\in U})
&=&\rho _{h^{-1}(U)h^{-1}(V)}^{\prime }\left( \left( \phi
_{ann_{S}(T)}(t_{ann_{R}(h(T))})\right) _{T\in h^{-1}(U)}\right) \\
&=&\left( \phi _{ann_{S}(T)}(t_{ann_{R}(h(T))})\right) _{T\in h^{-1}(V)} \\
&=&h^{\#}(V)((t_{ann_{R}(P)})_{P\in V})=h^{\#}(V)\rho
_{UV}((t_{ann_{R}(P)})_{P\in U})
\end{eqnarray*}%
and hence $\rho _{h^{-1}(U)h^{-1}(V)}^{\prime }h^{\#}(U)=h^{\#}(V)\rho _{UV}$
for every open subset $U$ of $Spec^{s}(M)$. So the above diagram is
commutative. It follows that $h^{\#}:\mathcal{O}(A,M)\longrightarrow h_{\ast
}(\mathcal{O}(B,N))$ is a morphism of sheaves.
\end{proof}

\begin{corollary}
\label{morphism2}Let $\phi :R\longrightarrow S$ be a ring homomorphism, let $%
N$ be an $S$-module and $M$ be a secondful $R$-module such that $Spec^{s}(M)$
is a $T_{0}$-space and $ann_{R}(M)\subseteq ann_{R}(N)$. Then $\phi $
induces a morphism of locally ringed spaces%
\begin{equation*}
(h,h^{\#}):(Spec^{s}(N),\mathcal{O}(S,N))\longrightarrow (Spec^{s}(M),%
\mathcal{O}(R,M))
\end{equation*}
\end{corollary}

\begin{proof}
Taking $A=R$, $B=S$ and $\delta =\phi $ in Theorem \ref{Theorem13}, we get
the morphism of sheaves $h^{\#}:\mathcal{O}(R,M)\longrightarrow h_{\ast }(%
\mathcal{O}(S,N))$ which is defined as in the proof of Theorem \ref%
{Theorem13}. By Theorem \ref{Stalk}, the map on the stalks $h_{T}^{\#}:%
\mathcal{O}(R,M)_{h(T)}\longrightarrow \mathcal{O}(S,N)_{T}$ is clearly the
local homomorphism%
\begin{equation*}
\begin{array}{cccc}
\phi _{ann_{S}(T)}: & R_{f(ann_{S}(T))} & \longrightarrow & S_{ann_{S}(T)}
\\ 
& \frac{r}{s} & \longrightarrow & \frac{\phi (r)}{\phi (s)}%
\end{array}%
\end{equation*}%
where $f$ is the map defined in the proof of Theorem \ref{Theorem13}. This
implies that 
\begin{equation*}
(h,h^{\#}):(Spec^{s}(N),\mathcal{O}(S,N))\longrightarrow (Spec^{s}(M),%
\mathcal{O}(R,M))
\end{equation*}%
is a locally ringed spaces.
\end{proof}

\bigskip

\textbf{Acknowledgement }

The authors would like to thank the Scientific Technological Research
Council of Turkey (TUBITAK) for funding this work through the project
114F381.

The second author was supported by the Scientific Research Project
Administration of Akdeniz University.

\textit{Se\c{c}il \c{C}eken: Trakya University, Faculty of Sciences,
Department of Mathematics, Edirne, Turkey.}

\textit{Mustafa Alkan: Akdeniz University, Faculty of Sciences, Department
of Mathematics, Antalya, Turkey.}

\end{document}